\documentclass[Threeside,11pt]{article}
\usepackage{etoolbox}
\makeatletter
\patchcmd{\thebibliography}
{\list}
{\small          
	\setstretch{1.2}
	\list}
{}{}
\makeatother
\usepackage{setspace}
\usepackage{tikz}
\usepackage{float}
\definecolor{myPurple}{RGB}{109,90,207}
\definecolor{myGreen}{RGB}{4,130,67}
\definecolor{myGold}{RGB}{253,177,71}
\definecolor{myBurgundy}{RGB}{63,1,44}
\definecolor{myTeal}{RGB}{8,119,127}
\definecolor{myCream}{RGB}{255,216,177}
\definecolor{myOrange}{RGB}{225,119,1}
\usetikzlibrary{arrows.meta,bending,calc}
\tikzset{
	ball3Donly/.style={
		circle, minimum size=1.4cm,
		shading=ball,
		ball color=myPurple,
		draw=myPurple
	}
}
\pgfmathsetmacro{\wNear}{2.0pt}
\pgfmathsetmacro{\wMid}{1.3pt}
\pgfmathsetmacro{\wFar}{0.7pt}
\pgfmathsetmacro{\bend}{6mm}
\usepackage[latin1]{inputenc}
\usepackage[english]{babel}
\usepackage{amsmath}
\usepackage{amsthm}
\usepackage{algorithm}
\usepackage{cite}
\usepackage{amsfonts}
\usepackage{amssymb}
\usepackage{mathrsfs}
\usepackage{graphicx}
\usepackage{colortbl,dcolumn}
\usepackage{marvosym}
\usepackage{ifsym}
\usepackage{paralist}
\usepackage{xcolor}
\usepackage{subcaption}
\usepackage{adjustbox}
\usepackage{enumitem}
\usepackage{booktabs,tabularx}
\usepackage[colorlinks,citecolor=blue,
linkcolor=blue,hypertexnames=false]{hyperref} 
\allowdisplaybreaks
\DeclareFontFamily{U}{mathx}{\hyphenchar\font45}
\DeclareFontShape{U}{mathx}{m}{n}{
	<5> <6> <7> <8> <9> <10>
	<10.95> <12> <14.4> <17.28> <20.74> <24.88>
	mathx10
}{}
\DeclareSymbolFont{mathx}{U}{mathx}{m}{n}
\DeclareMathAccent{\widecheck}{0}{mathx}{"71}
\usepackage{graphicx}
\usepackage{caption}
\usepackage{amsthm}

\newtheorem{lemma}{\bf Lemma}[section]

\newtheorem{proposition}{\bf Proposition}[section]

\newtheorem{remark}{\bf Remark}[section]

\numberwithin{equation}{section}

\allowdisplaybreaks
\newtheorem{theorem}{Theorem}[section]

\theoremstyle{remark}

\newcommand{\Z}{\mathbb Z}\newcommand{\R}{\mathbb R}
\newcommand{\T}{\mathbb T}
\newcommand{\la}[1]{\langle #1\rangle}
\DeclareMathOperator{\dist}{dist}

\allowdisplaybreaks
\theoremstyle{remark}

\begin{document}
	\title{{\sl Nonpersistence of full-dimensional tori with polynomial decay in nonlinear Schr\"odinger equations
}}
\author{Zhicheng Tong\thanks{School of Mathematics, Jilin University, Changchun 130012, P. R.  China. Email: \url{tongzc25@jlu.edu.cn}} 
	\and 
	Yong Li\thanks{Corresponding author. School of Mathematics, Jilin University, Changchun 130012, P. R.  China;   Center for Mathematics and Interdisciplinary Sciences,  Northeast  Normal  University, Changchun 130024, P. R. China. Email: \url{liyong@jlu.edu.cn}}
}

\date{}
\maketitle

\begin{abstract}
{
We establish local nonpersistence of full-dimensional invariant tori with polynomially decaying actions for the periodic quintic nonlinear Schr\"odinger equation
\[
iu_t-u_{xx}+T_vu+\mu|u|^4u=0,
\quad x\in\T:=\R/(2\pi\Z).
\]
For every $\kappa\geqslant1$, we construct a fixed bounded real Fourier multiplier $T_ve^{ijx}=v_je^{ijx}$ whose linear frequencies $\lambda_j=j^2+v_j$ satisfy the Bourgain-type infinite-dimensional Diophantine condition
\[
\dist\left(\sum_{j\in\Z}\ell_j\lambda_j,\Z\right)
\geqslant\gamma_\kappa
\prod_{j\in\Z}\left(1+\ell_j^2\langle j\rangle^{2\kappa}\right)^{-1},
\quad \langle j\rangle:=\max\{1,|j|\},
\]
for some $\gamma_\kappa>0$ and every nonzero finitely supported integer sequence $\ell$. For this multiplier, every $3<M<4\kappa+\frac12$, every $A>0$, and all sufficiently small $\mu A^2>0$, the closed relative action neighbourhood
\[
\left\{
u=\sum_{j\in\Z}q_je^{ijx}\in H^1(\T):
\sup_{j\in\Z}
\left|\frac{|q_j|^2}{A\langle j\rangle^{-M}}-1\right|
\leqslant\frac3{10}
\right\}
\]
contains no full-dimensional invariant torus carrying a complete flow, with arbitrary frequencies allowed. This advances the study of the polynomial-decay problem for full-dimensional invariant tori, singled out by Kuksin [Amer. Math. Soc. Transl. Ser.~2, 212 (2004), 237--258, Problem~7.1] as a central open problem in Hamiltonian PDEs and also highlighted by Bourgain [Russian Math. Surveys 59 (2004), 231--246, Section~3]. The proof combines a multiplier construction ensuring both the Diophantine condition and weighted homological estimates with an exact common-phase transformation and a single normal-form step on four adjacent Fourier coordinates. A uniform phase-rotation estimate forces an action change beyond the prescribed neighbourhood, while overlapping parameter intervals yield nonpersistence for every sufficiently small positive perturbation parameter. The resonant mechanism and weighted estimates suggest possible applications to nonpersistence problems for other Hamiltonian PDEs.\\
        \\
        {\bf Keywords:} {Full-dimensional invariant tori, polynomial action decay, local nonpersistence, nonlinear Schr\"odinger equations, infinite-dimensional KAM theory}
        \vspace{2mm}
        \\
        {\bf2020 Mathematics Subject Classification:} {37K55, 
        35Q55, 
        37K45, 
        35B15 
}
\par}
\end{abstract}

\section{Introduction}
\textit{What is the fate of full-dimensional invariant tori with polynomially decaying positive actions under perturbations in Hamiltonian PDEs?} In 2004, Kuksin \cite[Problem~7.1]{Kuk04} identified this as the main open problem in the field. In the same year, Bourgain \cite[Section~3]{Bou04} likewise emphasized that this fundamental question remained open (see Section~\ref{SUBBJ} for historical details). Such tori describe almost-periodic motions in which infinitely many Fourier modes remain active, with actions of order $j^{-M}$.

We advance this long-standing problem by establishing local nonpersistence for the quintic nonlinear Schr\"odinger equation (NLS) on the circle. Specifically, we construct a fixed bounded Fourier multiplier for which a prescribed relative action tube around the polynomial product torus contains no full-dimensional invariant torus for any sufficiently small positive perturbation parameter.

\subsection{Main result}

Let $\T:=\R/(2\pi\Z)$ denote the standard torus, and let $\la j:=\max\{1,|j|\}$. We normalize spatial integration by $dx/(2\pi)$ and consider the following quintic NLS {with periodic boundary conditions}
\begin{equation}\label{eq:nls}
 iu_t-u_{xx}+T_vu+\mu|u|^4u=0,
 \quad {x\in\T},\quad T_ve^{ijx}=v_je^{ijx},
\end{equation}
where $ T_v $ is a multiplier, with  $v=(v_j)_{j\in\Z}$ being real and bounded. {The spatial domain $\T$ means that $u(x+2\pi,t)=u(x,t)$.} Using
$u(x,t)=\sum_{j\in\Z}q_j(t)e^{ijx}$ and $\lambda_j=j^2+v_j$, the Fourier equations read
\begin{equation}\label{eq:fourier}
 \dot q_j=i\lambda_jq_j+i\mu N_j(q),\quad
 N_j(q)=\sum_{n_1,n_2,n_3,n_4,n_5\in\Z,\ n_1+n_2+n_3-n_4-n_5=j}
 q_{n_1}q_{n_2}q_{n_3}\bar q_{n_4}\bar q_{n_5}.
\end{equation}
We work in the Banach space
\[
 h^1=\left\{q:\|q\|_{h^1}^2:=\sum_{j\in\Z}\la j^2|q_j|^2<+\infty\right\}.
\]
A forward complete trajectory is a continuous map
$q:[0,\infty)\to h^1$ whose coordinates satisfy \eqref{eq:fourier}.

For $M>3$ and $A>0$, the unperturbed equation leaves invariant the product torus
\begin{equation*}
 \mathcal T_{M,A}:=
 \left\{q_j=\sqrt A\,\la j^{-M/2}e^{i\theta_j}:
               \theta\in\T^{\Z}\right\}\subset h^1.
\end{equation*}
Indeed, its flow is $q_j(t)=e^{i\lambda_jt}q_j(0)$, which preserves every action $|q_j|^2$. We measure preservation of this profile by the \textit{closed relative action tube}
\begin{equation}\label{eq:tube}
 \mathcal U_{M,A}=
 \left\{q\in h^1:
 \sup_{j\in\Z}\left|
       \frac{|q_j|^2}{A\la j^{-M}}-1\right|\leqslant\frac3{10}\right\},
\end{equation}
which has empty interior in $h^1$. Below, we write $\Z^{\Z}_*$ for the integer sequences with finite support.

\begin{theorem}\label{thm:main}
For every fixed $\kappa\geqslant1$, there are a real sequence
$v=v^{(\kappa)}$ with $\|v\|_{\ell^\infty}\leqslant11/10$ and a constant $\gamma_\kappa>0$ such that the following hold.

First, for every $0\ne\ell\in\Z^{\Z}_*$, we have the infinite-dimensional Diophantine property
\begin{equation}\label{eq:DC}
 \dist\left(\sum_{j\in\Z}\ell_j\lambda_j,\Z\right)
 \geqslant\frac{\gamma_\kappa}{D_\kappa(\ell)},\quad
 D_\kappa(\ell):=\prod_{j\in\Z}\left(1+\ell_j^2\la j^{2\kappa}\right).
\end{equation}
Second, for every
\begin{equation}\label{eq:Mrange}
 3<M<4\kappa+\frac12,
\end{equation}
there is $\mu_*(\kappa,M)>0$ such that, for every fixed $A>0$ and
\begin{equation}\label{eq:smallparameter}
 0<\mu <A^{-2}\mu_*(\kappa,M),
\end{equation}
equation \eqref{eq:nls} has no forward complete trajectory contained in $\mathcal U_{M,A}$. In particular, $\mathcal U_{M,A}$ contains no invariant torus carrying a complete flow.
\end{theorem}

We remark that the multiplier $ T_v $ and the Diophantine constant $ \gamma_\kappa $ are chosen after $\kappa$ and before $M,A,\mu$. Therefore, the same multiplier  works for every exponent $ M $ in \eqref{eq:Mrange} and every amplitude $A>0$. 
In particular, $\kappa=1$ yields one multiplier satisfying
\[
 \dist(\ell\cdot\lambda,\Z)
 \geqslant\gamma_1\prod_{j\in\Z}(1+\ell_j^2\la j^2)^{-1}
\]
and gives the conclusion simultaneously for $3<M<9/2$.
For any prescribed finite $M>3$, one can choose a finite
$\kappa\geqslant1$ with $\kappa>(M-1/2)/4$.

The notion of persistence considered here is uniform convergence of the relative actions. More precisely, suppose that $K_\mu:\T^{\Z}\to h^1$ parameterizes an invariant torus carrying a complete flow and that
\[
 \sup_{\theta\in\T^{\Z}}\sup_{j\in\Z}
 \left|\frac{|K_{\mu,j}(\theta)|^2}{A\la j^{-M}}-1\right|
 \to0,\quad \mu\downarrow0.
\]
For sufficiently small $\mu$, its image would lie in $\mathcal U_{M,A}$, contradicting Theorem~\ref{thm:main}. The frequencies of the motion on the perturbed torus are unrestricted. This relative action condition specifies the topology of persistence. The conclusion is local to the prescribed relative action tube: it does not exclude invariant tori satisfying two-sided polynomial action bounds with larger comparison constants, or persistence measured only in the $h^1$ norm.

\begin{remark}
This obstruction to persistence concerns relative Fourier actions and is compatible with stability of the zero solution in the energy space. As an elementary illustration, independent of equation~\eqref{eq:nls}, consider the two-mode system
\[
 \dot z_-=iz_+,\quad \dot z_+=iz_-,\quad
 u(t,x)=z_-(t)e^{-ix}+z_+(t)e^{ix}.
\]
The quantity $|z_-|^2+|z_+|^2$ is conserved and controls the $H^1$ norm of $u$, so the zero solution is stable. Nevertheless, the solution $z_-(t)=\varepsilon\cos t$, $z_+(t)=i\varepsilon\sin t$ satisfies
\[
 |z_-(t)|^2=\varepsilon^2\cos^2t,\quad
 |z_+(t)|^2=\varepsilon^2\sin^2t.
\]
Between $t=\pi/4$ and $t=\pi/2$, each action changes by $\varepsilon^2/2$, equal to its value at $t=\pi/4$, while the $H^1$ norm remains constant. Thus a change comparable to an individual action does not imply instability of zero in $H^1$. For equation~\eqref{eq:nls} itself, the conserved mass $\mathcal M(u)=\|u\|_{L^2}^2$ and energy
\[
 E(u)=\|u_x\|_{L^2}^2+\langle T_vu,u\rangle_{L^2}
       +\frac{\mu}{3}\|u\|_{L^6}^6
\]
satisfy $E(u)+(\|v\|_{\ell^\infty}+1)\mathcal M(u)\geqslant\|u\|_{H^1}^2$ for $\mu>0$, with the standard norm $\|u\|_{H^1}^2=\|u\|_{L^2}^2+\|u_x\|_{L^2}^2$. Together with $H^1(\T)\subset L^6(\T)$, this yields stability of zero in $H^1$, consistently with the relative action conclusion of Theorem~\ref{thm:main}.
\end{remark}

\begin{remark}
For each $\kappa\geqslant1$, condition~\eqref{eq:DC} holds for a set of full product Lebesgue measure in $[0,1]^{\Z}$ when the positive constant $\gamma_\kappa$ is allowed to depend on $v$. For a prescribed $\gamma_\kappa$, the exceptional set has measure at most $C_\kappa\gamma_\kappa$, by the standard product estimate underlying the conditions in~\cite{Bou05}: each resonant strip has measure at most $2\gamma_\kappa/D_\kappa(\ell)$, and $\sum_{\ell\in\Z_*^{\Z}\setminus\{0\}}D_\kappa(\ell)^{-1}<+\infty$. Thus this arithmetic class is not restricted to frequencies with exceptionally small divisors; it also contains frequencies satisfying stronger lower bounds that imply~\eqref{eq:DC}. In particular, $\kappa=2$ gives the product condition used in~\cite{CLSY18,Con24}, whereas the condition with $\kappa=1$ is stronger and already yields the conclusion for $3<M<9/2$. 

For the same periodic quintic NLS model, full-dimensional tori have been constructed, with suitable choices of the multiplier and corresponding Bourgain-type arithmetic assumptions, at the subexponential scales $I_j\asymp e^{-c|j|^\theta}$ for $0<\theta<1$ in~\cite{Bou05,CLSY18}, and at the slower scale $I_j\asymp e^{-c(\log|j|)^\sigma}$ for $\sigma>2$ in~\cite{Con24}; see also~\cite{BMP21} for a Gevrey framework. Both scales decay faster than every power. In contrast, Theorem~\ref{thm:main} gives an actual dynamical obstruction at polynomial decay for a fixed multiplier satisfying~\eqref{eq:DC}. A general persistence statement in the present relative action topology therefore cannot be extended to polynomial decay under this arithmetic assumption alone merely by improving the existing KAM estimates. 
\end{remark}

\subsection{Background}\label{SUBBJ}

{
The celebrated KAM theory for Hamiltonian PDEs studies the persistence of recurrent motions when infinitely many linear oscillators are coupled by a nonlinearity. Kuksin \cite{Kuk87} established a perturbation theory for finite-dimensional invariant tori in an infinite-dimensional phase space, and Kuksin and P\"oschel \cite{KP96} constructed invariant Cantor manifolds of quasi-periodic solutions for NLSs. Kuksin's survey \cite{Kuk04} places these results alongside perturbations of integrable PDEs, small-amplitude solutions and long-time stability. The passage to infinite-dimensional tori requires further control: the smallness thresholds in finite-dimensional KAM constructions generally depend on the dimension, so a direct limit need not produce a nontrivial almost-periodic motion.

The decay of the infinitely many nonzero actions is central to this passage. In 2004, Kuksin \cite[Problem~7.1]{Kuk04} identified \textit{the fate of full-dimensional invariant tori with positive actions of order $j^{-M}$} as a central open problem in the theory of one-dimensional Hamiltonian PDEs. In the same year, Bourgain \cite[Section~3]{Bou04} likewise highlighted that this fundamental question remained open.  Yuan \cite[Problem~1]{Yua14} subsequently formulated the question explicitly in his survey of short-range Hamiltonian systems. More recently, Biasco, Massetti, and Procesi \cite{BMP23} explicitly connected the persistence of Sobolev almost-periodic solutions with Kuksin's Problem~7.1, describing their construction based on sparse tangential modes as a first step in this direction. Together, these works identify polynomially decaying actions as a natural setting in which to test infinite-dimensional persistence.

For the periodic quintic NLS with a convolution potential, Bourgain \cite{Bou05} treated infinitely many modes within a single construction and obtained full-dimensional tori with $I_j\asymp e^{-c\sqrt{|j|}}$. Cong, Liu, Shi, and Yuan \cite{CLSY18} proved long-time stability of these tori. Their Theorem~1.1 and Remark~1.2 also establish the extension to $I_j\asymp e^{-c|j|^\theta}$ for every $0<\theta<1$. Cong, Mi, Shi, and Wu \cite{CMSW19} extended this existence theory to a quintic nonlinearity $f(x)|u|^4u$ with Gevrey dependence on $x$.

Building on Bourgain's use of the original complex Fourier coordinates, Biasco, Massetti, and Procesi \cite{BMP21} proved a counterterm theorem of Herman type with a KAM construction uniform in the dimension of the torus. Their framework treats finite- and infinite-dimensional tori, allows actions in a small ball of a weighted Gevrey space without imposing a positive lower bound on every action, and covers nonlinearities depending on $x$. This enlarges the admissible families of tori while retaining a Gevrey upper bound on their Fourier amplitudes. The role of the active modes becomes particularly clear in their subsequent Sobolev construction \cite{BMP23}: the independent angular variables are associated with a sparse set of tangential Fourier modes, and the resulting almost-periodic solutions may have only finite Sobolev regularity, including weak solutions.

The distinction between sparse and maximal tori is discussed in Corsi, Gentile, and Procesi \cite[Section~3]{CGP24}. Their own main theorem proves persistence of maximal tori for a mechanical system of infinitely many weakly interacting rotators, using a renormalisation group expansion and weighted analyticity assumptions. Their discussion of the PDE literature identifies the Sobolev tori of \cite{BMP23} as non-maximal and leaves open the corresponding maximal Sobolev problem. A related advance by Bernier, Gr\'ebert, and Robert \cite{BGR24} removes external parameters for a class of periodic NLS equations with a nonzero cubic term; their infinite-dimensional tori accumulate on finite-dimensional KAM tori and are again built using sparse sets of modes.

The choice of external parameters provides another connection between these constructions. For the quintic NLS, Corsi, Gentile, and Procesi \cite{CGP23} obtained spatially Gevrey almost-periodic solutions for a large set of convolution potentials whose Fourier coefficients decay polynomially, with frequencies satisfying a Bryuno condition. The polynomial decay here concerns the potential coefficients; the solutions remain Gevrey. Biasco, Corsi, Gentile, and Procesi \cite{BCGP24} further constructed families for which the sets of admissible potentials and Gevrey initial data have asymptotically full probability measure as the nonlinear parameter tends to zero. In suitable regions with positive actions, their maximal tori are Banach submanifolds forming a Cantor foliation. These results strengthen control of the parameters and the size of the family of persistent motions.

Let us return to the basic polynomial-decay question raised by Kuksin and Bourgain. Along the line of results retaining an independent angle at every Fourier mode, the admissible action decay has been progressively slowed. For the periodic quintic NLS with a bounded convolution multiplier, the slowest explicit decay known to us is the scale reached by Cong \cite{Con24}:
\[
 I_j\asymp\exp\bigl(-c(\log|j|)^\sigma\bigr),
 \quad \sigma>2,\quad |j|\to+\infty.
\]
Wu \cite{Wu25} subsequently extended this scale to $f(x)|u|^4u$, with the Fourier coefficients of $f$ satisfying a compatible logarithmic decay bound. A natural direction for further refinement is to investigate the borderline exponent $2$ with an additional iterated-logarithmic weight, for example
\[
 I_j\asymp\exp\!\left[-c(\log|j|)^2
                    (\log\log|j|)^{1+\eta}\right],
 \quad \eta>0,\quad |j|\to+\infty.
\]
Such a refinement would still give decay faster than every power. Reaching polynomial decay represents a more substantial change in spatial regularity. Indeed, for $I_j=|q_j|^2\asymp\la j^{-M}$,
\[
 \|u\|_{H^s(\T)}^2\asymp
 \sum_{j\in\Z}\la j^{2s}I_j
 \asymp\sum_{j\in\Z}\la j^{2s-M}<+\infty
 \quad\Leftrightarrow\quad s<\frac{M-1}{2}.
\]
Thus a two-sided polynomial action bound places the torus at a finite Sobolev regularity threshold, whereas the logarithmic decay scales above give spatially $C^\infty$ functions.

This change suggests examining persistence from a regularity perspective. In a related infinite-dimensional KAM problem, the authors established $C^1$ linearization of $C^\infty$ vector fields on $\T^{\Z}$, under weighted spatial assumptions and suitable nonresonance conditions, with a constant modifying term \cite{TL23}. The construction reaches non-Gevrey $C^\infty$ classes through a KAM iteration with a balancing sequence. On the converse side, Cheng and Wang \cite[Theorem~1.1]{CW13} showed that, for $H_0(y)=\frac12\sum_{i=1}^d y_i^2$ with $d\geqslant2$, a Lagrangian torus with any prescribed rotation vector can be destroyed by $C^\infty$ perturbations arbitrarily small in $C^{2d-\delta}$, for every small $\delta>0$. Since the threshold grows with $d$, no fixed finite differentiability order can yield a persistence theorem uniform over all dimensions in this general Hamiltonian class. These results concern differentiability of the perturbing vector field or Hamiltonian, while polynomial action decay concerns spatial regularity of the PDE solution. \textit{Taken as a main motivation, this distinction leads us to ask whether polynomially decaying full-dimensional tori can fail to persist locally even for an analytic nonlinearity, when proximity is measured in relative actions.}

Theorem~\ref{thm:main} gives such a local nonpersistence result for the quintic NLS: a single bounded multiplier satisfying \eqref{eq:DC} excludes full-dimensional invariant tori throughout the fixed tube \eqref{eq:tube}, for all sufficiently small positive $\mu$, with arbitrary frequencies allowed on a candidate torus. The obstruction is a quantitative exchange of actions among four adjacent high Fourier modes. Resonant energy exchange in the quintic NLS was studied by Gr\'ebert and Thomann \cite{GT12}, who constructed periodic action exchanges in an effective four-mode system and proved that solutions of the full PDE follow this dynamics over long finite times. Here the four-mode calculation yields an action excursion uniform in the initial phases and in all remaining Fourier coordinates satisfying the action bounds. This provides a concrete nonpersistence mechanism for the polynomial-decay problem. The use of a finite resonant interaction, an exact phase transformation and weighted homological estimates also suggests a route to analogous questions for other Hamiltonian PDEs with suitable resonances.
\par}

\subsection{Proof strategy}

The proof turns the confinement of a hypothetical trajectory in $\mathcal U_{M,A}$ into a finite-time action excursion exceeding the width of the tube. The excursion is measured relative to the reference action at the selected high-frequency modes, whose absolute size tends to zero with the perturbation parameter. The construction combines three ingredients.
\begin{itemize}
\item A fixed multiplier provides small detunings on infinitely many disjoint four-mode blocks while satisfying the global condition~\eqref{eq:DC}. Affine relations prescribe the detunings, and independent random variables in the remaining directions yield the weighted inverse estimates for the homological equation. A simultaneous selection makes these estimates available for all required decay exponents.
\item An exact common-phase transformation cancels every monomial containing exactly one positive occurrence of the output coordinate. A single normal-form step on the four selected output coordinates then isolates the interaction with remainder bounded by $C\mu^2m^{-1}\la j^{-M/2}$. Finite monomial truncation and integration implement this transformation along the trajectory. The estimates are uniform over the remaining Fourier coordinates subject to the tube bounds.
\item Monotone rotation of the relative phase produces a fixed action excursion over a half-turn, uniformly in the initial phase. The fine geometric spacing of the blocks makes the associated parameter intervals overlap. Thus the same multiplier yields the obstruction for every sufficiently small positive perturbation parameter $\mu$.
\end{itemize}
For $A=1$, conservation of the mass $S_*$ fixes the coefficient of the leading interaction. At a block with centre $m=m_k$, the action, interaction, and detuning scales satisfy
\[
 I=m^{-M},\quad \beta=\mu S_*I,
 \quad |\Delta_k|\asymp(k+1)^2m^{-8\kappa},
 \quad \beta\asymp|\Delta_k|.
\]
The remainder divided by $\beta\sqrt I$ is bounded by
\[
 C\mu m^{M-1}\leqslant C'(k+1)^2m^{2M-8\kappa-1}.
\]
Its decay gives the range~\eqref{eq:Mrange} and closes the four-mode estimate. The numerical constants in the phase and overlap estimates provide explicit margins for the action excursion and the coverage of the perturbation parameters.

Section~\ref{sec:mass} establishes mass conservation and the interaction scale. Section~\ref{sec:multiplier} specifies the multiplier and states the inverse estimates used in the four-coordinate reduction of Section~\ref{sec:normalform}. Section~\ref{sec:phase} proves the uniform escape estimate, and Section~\ref{sec:completion} applies it on overlapping parameter intervals and rescales the amplitude. Finally, Section~\ref{sec:auxiliary} proves the weighted summation estimates and completes the simultaneous selection of the multiplier.

\section{Conserved mass and the interaction scale}\label{sec:mass}

We take $A=1$ until the final amplitude rescaling. The total mass determines the scale of the four-mode interaction. We first prove that this mass is constant, then identify the action and time scales needed to force an excursion beyond the tube.

\subsection{A hypothetical trajectory and conservation of mass}

We rule out a forward complete trajectory contained in $\mathcal U_{M,1}$; this also excludes an invariant torus carrying a complete flow.
Fix $M>3$ and set
\begin{equation*}
 p=\frac M2>\frac32,\quad
 I_j^0=\la j^{-M},\quad
 Z_M=\sum_{j\in\Z}I_j^0=1+2\zeta(M).
\end{equation*}
Suppose, to obtain a contradiction, that $q(t)$ is such a trajectory. For every $t\geqslant0$,
\begin{equation*}
 0.7I_j^0\leqslant|q_j(t)|^2\leqslant1.3I_j^0,
 \quad 0.7Z_M\leqslant\sum_{j\in\Z}|q_j(t)|^2\leqslant1.3Z_M.
\end{equation*}
In particular, $q(t)$ is uniformly bounded both in $h^1$ and in
\begin{equation*}
 X_p=\{q:\|q\|_p:=\sup_j\la j^p|q_j|<+\infty\},
 \quad \|q(t)\|_p\leqslant\sqrt{1.3}.
\end{equation*}

We establish mass conservation by differentiating finite coordinate sums and passing to the limit in their integrated identities. For every $n$, we have
\[
 \frac d{dt}\sum_{j\in\Z,\ |j|\leqslant n}|q_j|^2
 =2\operatorname{Re}\sum_{j\in\Z,\ |j|\leqslant n}
      \bar q_j(i\lambda_jq_j+i\mu N_j(q))
 =-2\mu\operatorname{Im}\sum_{j\in\Z,\ |j|\leqslant n}\bar q_jN_j(q).
\]
The linear contribution vanishes coordinatewise because $\lambda_j$ is real.
Integration gives
\begin{equation}\label{eq:mass-truncation}
 \sum_{j\in\Z,\ |j|\leqslant n}\bigl(|q_j(t)|^2-|q_j(0)|^2\bigr)
 =-2\mu\int_0^t\operatorname{Im}
       \sum_{j\in\Z,\ |j|\leqslant n}\bar q_j(s)N_j(q(s))\,ds.
\end{equation}
The convolution bounds give
\[
 \|q\|_{\ell^1}\leqslant
 \left(\sum_{j\in\Z}\la j^{-2}\right)^{1/2}\|q\|_{h^1},
 \quad
 \|N(q)\|_{\ell^2}\leqslant\|q\|_{\ell^1}^4\|q\|_{\ell^2}.
\]
Then
\[
 \left|\sum_{j\in\Z,\ |j|>n}\bar q_jN_j(q)\right|
 \leqslant
 \left(1.3\sum_{j\in\Z,\ |j|>n}\la j^{-M}\right)^{1/2}\|N(q)\|_{\ell^2}
 \leqslant C_M n^{(1-M)/2},
\]
uniformly in time, so we can let $n\to+\infty$ in
\eqref{eq:mass-truncation}. Absolute convergence of the Fourier series also yields
\[
 \sum_{j\in\Z}\bar q_jN_j(q)
 =\int_{\T}\bar u\,|u|^4u\,\frac{dx}{2\pi}
 =\int_{\T}|u|^6\frac{dx}{2\pi}\in\R,
\]
showing that the mass is constant. We write
\begin{equation}\label{eq:mass}
 S_*:=\sum_{j\in\Z}|q_j(t)|^2,\quad
 S_-:=0.7Z_M\leqslant S_*\leqslant1.3Z_M=:S_+.
\end{equation}
The fixed interval $[S_-,S_+]$ will yield a common small-parameter threshold for all hypothetical trajectories.

\subsection{Four adjacent modes}

The full width of the $j$th action interval is $0.6I_j^0$. To produce a change on this scale, we use four adjacent modes with comparable reference actions and set
\[
 a=m,\quad b=m+1,\quad c=m+2,\quad d=m+3.
\]
Their reference actions are asymptotic to $I=m^{-M}$, since
\[
 \frac{I_{m+r}^0}{I}=\left(\frac m{m+r}\right)^M\to1,
 \quad r=0,1,2,3.
\] The monomial
$q_aq_d\bar q_b\bar q_c$ satisfies the momentum relation $a+d=b+c$.
If $q_j=\sqrt{I_j}e^{i\theta_j}$, its phase is
\[
 \Theta=\theta_a+\theta_d-\theta_b-\theta_c,
\]
and the associated linear detuning is
\begin{equation*}
 \Delta=\lambda_a+\lambda_d-\lambda_b-\lambda_c.
\end{equation*}
Multiplication by an additional action gives a term of degree six in the Hamiltonian of the quintic equation. Summing over that additional index produces a coefficient comparable to $S_*$. Since
\[
 |q_aq_d\bar q_b\bar q_c|\asymp I^2,
\]
its contribution to an action derivative has size $\mu S_*I^2$. Dividing by the action scale $I$ gives the relative exchange rate
\begin{equation*}
 \beta=\mu S_*I.
\end{equation*}

The action derivative is proportional to $\sin\Theta$. We choose the detuning to dominate the nonlinear phase correction and force $\Theta$ to rotate, while keeping it comparable to $\beta$ so that a half-turn lasts long enough to produce the required change. The choice $-\Delta\approx23.2\beta$ achieves both bounds, as proved in Section~\ref{sec:phase}.

A half-turn then lasts $\mathcal{O}(\beta^{-1})$, so a derivative of size $\beta I$ produces an action change of size $I$. A remainder $R_j$ in the coordinate equation contributes at most $2|q_j||R_j|$ to the action derivative and at most $|R_j|/|q_j|$ to the phase derivative. Since $|q_j|\asymp\sqrt I$, both errors are small on this scale precisely when
\begin{equation}\label{eq:required-error}
 |R_j|=o(\beta\sqrt I).
\end{equation}
We now construct one multiplier with the required detunings and isolate the selected interaction with the error bound \eqref{eq:required-error}.

\section{A fixed multiplier and weighted homological estimates}\label{sec:multiplier}

We prescribe detunings on disjoint blocks of one fixed multiplier. For each $\mu$ and conserved mass $S_*$, we will select a block on which $-\Delta\asymp\mu S_*m^{-M}$. The same multiplier must provide weighted inverse bounds for removing the remaining interactions. Proposition~\ref{prop:fixedpotential} records these requirements; its proof is given in Section~\ref{sec:potential}.

\subsection{The block construction}

The Diophantine denominator in \eqref{eq:DC} has size $m^{8\kappa}$ for the four-mode exponent vector. We choose the detunings on the inverse scale, with a factor $(k+1)^2$ ensuring summability of the exceptional probabilities. A fine geometric sequence of block centres will give overlapping parameter intervals. Fix $\kappa\geqslant1$ and put
\begin{equation}\label{eq:centers}
 \begin{gathered}
 Q_\kappa=\exp\!\left(\frac1{100\kappa}\right),\quad
 m_k=\lceil LQ_\kappa^k\rceil,\quad k\geqslant1,\\
 T_k=\{m_k,m_k+1,m_k+2,m_k+3\},\quad
 \tau_k=(k+1)^2m_k^{-8\kappa}.
 \end{gathered}
\end{equation}
We choose the integer $L$, depending only on $\kappa$, so large that the blocks are pairwise disjoint,
\[
 m_k\geqslant100,\quad 1<\frac{m_{k+1}}{m_k}\leqslant1.02,
 \quad \tau_k\leqslant\frac1{100}.
\]
Let $e_j$ denote the standard coordinate vector and set
\[
 \sigma^k=e_{m_k}-e_{m_k+1}-e_{m_k+2}+e_{m_k+3}.
\]
The kinetic contribution to the detuning is the fixed integer
\[
 m_k^2-(m_k+1)^2-(m_k+2)^2+(m_k+3)^2=4.
\]
We cancel this integer by a bounded choice of the four multiplier coefficients, retaining three random directions of fixed width and one direction on the detuning scale. Specifically, we prescribe
\begin{equation}\label{eq:potential}
 \begin{aligned}
 v_{m_k}&=-1+X_k,&v_{m_k+1}&=1+Y_k,\\
 v_{m_k+2}&=1+Z_k,&
 v_{m_k+3}&=-1-X_k+Y_k+Z_k-\tau_k+U_k.
 \end{aligned}
\end{equation}
Here $X_k,Y_k,Z_k$ are uniform on $[-1/100,1/100]$,
$U_k$ is uniform on $[-\tau_k/100,\tau_k/100]$, and the coefficients $v_j$ outside the blocks are uniform on $[-1/100,1/100]$. All these variables are independent. Every realization satisfies
\begin{equation}\label{eq:detuning}
 \|v\|_\infty\leqslant1.0401<\frac{11}{10},\quad
 \Delta_k:=\sigma^k\cdot\lambda=-\tau_k+U_k
       \in[-1.01\tau_k,-0.99\tau_k].
\end{equation}
The multiplier contribution is $-4-\tau_k+U_k$: the variables $X_k,Y_k,Z_k$ cancel from the selected detuning and control the other combinations in the homological equation.

The resulting operator is a bounded self-adjoint Fourier multiplier on $L^2(\T)$ and satisfies
\[
 \|T_vu\|_{H^s}\leqslant\|v\|_{\ell^\infty}\|u\|_{H^s},
 \quad s\in\R.
\]

\subsection{The inverse sums needed for the reduction}

We retain the selected interaction and remove the remaining quintic terms from the four output equations to first order in $\mu$. Solving the homological equation divides each term by its frequency difference; a weighted absolute sum will ensure convergence at the polynomial decay of the trajectory.

The common-phase change in Section~\ref{sec:normalform} first replaces $N$ by the polynomial $P$ defined by
\begin{equation}\label{eq:gaugepoly}
 \mathscr A(q)=\int_{\T}|u|^4\frac{dx}{2\pi},\quad
 P_j(q)=N_j(q)-3\mathscr A(q)q_j.
\end{equation}
For a quintic monomial with output $j$, write
\begin{equation}\label{eq:monomial}
 M_\nu(q)=q_{n_1}q_{n_2}q_{n_3}\bar q_{n_4}\bar q_{n_5},
 \quad n_1+n_2+n_3-n_4-n_5=j,
\end{equation}
and define its integer exponent vector and divisor by
\begin{equation*}
 \ell_\nu=e_{n_1}+e_{n_2}+e_{n_3}-e_{n_4}-e_{n_5}-e_j,
 \quad \Delta_\nu=\ell_\nu\cdot\lambda.
\end{equation*}
Thus $\ell_\nu$ records the signed multiplicity of each coordinate relative to the output coordinate.
For $j\in T_k$, we retain the zero exponent vector, which supplies the nonlinear frequency corrections, and the vectors $\pm\sigma^k$, which supply the chosen action exchange. Their divisors are zero or on the scale of the intended interaction. Let $P_j^0$ denote their sum and define
\begin{equation}\label{eq:decomposition}
 P_j=P_j^0+P_j^1,\quad
 P_j^1=\sum_{\nu\in\mathcal R_j}c_\nu M_\nu.
\end{equation}
We combine identical monomials; their coefficients are uniformly bounded because the degree is fixed. Here $j$ labels the output equation and the five indices $n_r$ label the input factors.

\begin{proposition}\label{prop:fixedpotential}
There is a realization of \eqref{eq:potential} satisfying \eqref{eq:DC} such that, for every real $p>3/2$, a finite constant $C(p)$ satisfies
\begin{equation}\label{eq:inverse}
 \sum_{\nu\in\mathcal R_j}
 \frac{|c_\nu|}{|\Delta_\nu|}
 \frac{\prod_{r=1}^5\la{n_r}^{-p}}{\la j^{-p}}
 \leqslant\frac{C(p)}{m_k},\quad j\in T_k,\quad k\geqslant1.
\end{equation}
The same realization works for every $p>3/2$.
\end{proposition}

If $|q_n|\leqslant R\la n^{-p}$, the product of input weights bounds the corresponding monomial, and division by $\la j^{-p}$ measures the correction relative to the output amplitude. Thus \eqref{eq:inverse} gives absolute convergence and a coordinate correction of relative size $\mathcal{O}(\mu/m_k)$. Section~\ref{sec:potential} proves the proposition by combining fractional inverse moments with the weighted summation estimate of Section~\ref{sec:window}.

We henceforth fix one realization supplied by Proposition~\ref{prop:fixedpotential}.

\section{Reduction to four equations with a controlled remainder}\label{sec:normalform}

We first remove the common-phase terms by an exact phase change. We then compute the retained interaction $P_j^0$ and construct a correction cancelling $P_j^1$ to first order in $\mu$. Finite monomial truncations justify this correction along the trajectory and yield the remainder estimate required by \eqref{eq:required-error}.

\subsection{An exact common-phase transformation}

To identify the terms requiring a common-phase change, consider the monomial
\[
 6q_jq_1q_{-1}\bar q_0^2.
\]
For large $j$, its divisor is $\lambda_1+\lambda_{-1}-2\lambda_0$ and its normalized weight in \eqref{eq:inverse} is one. Its inverse therefore gives a contribution independent of $j$. The subtraction $3\mathscr A(q)q_j$ cancels its coefficient, since $q_1q_{-1}\bar q_0^2$ has coefficient two in $\mathscr A$. As $\mathscr A(q)$ is real, we realize this subtraction by a common-phase change preserving every action.

We therefore set
\[
 q(t)=e^{i\phi(t)}z(t),\quad
 \dot\phi(t)=3\mu\mathscr A(q(t)),\quad \phi(0)=0.
\]
There are three positive and two conjugate factors in $N$, giving
\[
 N(e^{i\phi}z)=e^{3i\phi-2i\phi}N(z)=e^{i\phi}N(z),
 \quad \mathscr A(e^{i\phi}z)=\mathscr A(z).
\]
Since $\dot q_j=e^{i\phi}(\dot z_j+i\dot\phi z_j)$, substitution into \eqref{eq:fourier} yields
\[
 \dot z_j+i\dot\phi z_j=i\lambda_jz_j+i\mu N_j(z).
\]
Using the definition of $P_j$ in \eqref{eq:gaugepoly}, we obtain the exact equations
\begin{equation}\label{eq:gaugedflow}
 \dot z_j=i\lambda_jz_j+i\mu P_j(z),\quad |z_j|^2=|q_j|^2.
\end{equation}

The same subtraction cancels every term with exactly one positive occurrence of the output coordinate. Fix the output $j$ and expand $u=q_je^{ijx}+u^{\ne j}$, where $u^{\ne j}$ omits the $j$th Fourier mode. For terms with $a$ positive occurrences of $q_j$ and $b$ conjugate occurrences of $\bar q_j$, the binomial prefactor in the expansion of $P_j$ is
\begin{equation}\label{eq:binomial}
 \left[\binom3a-3\binom2{a-1}\right]\binom2b,
\end{equation}
with binomial coefficients outside their natural ranges equal to zero. The first term chooses the occurrences from $u^3\bar u^2$. In the subtracted term $3q_j\int u^2\bar u^2\,dx/(2\pi)$, one positive occurrence is already supplied by the outer $q_j$, so only $a-1$ positive occurrences remain to be chosen. For $a=1$, the bracket is $3-3=0$. Thus every surviving monomial has either no positive output factor or at least two. Section~\ref{sec:potential} uses this distinction to estimate the inverse sums.

\subsection{The retained terms}

In the decomposition \eqref{eq:decomposition}, $P_j^0$ retains the zero exponent vector and $\pm\sigma^k$. The first has zero divisor. At the tuned scale, the terms with exponents $\pm\sigma^k$ have size $S_*I^{3/2}$ and divisors of size $\beta=\mu S_*I$, so their inversion would produce a correction of size
\[
 \frac{\mu S_*I^{3/2}}{|\Delta_k|}\asymp\sqrt I,
\]
comparable with the coordinate itself. We retain them as the interaction driving the action excursion and compute $P_j^0$ explicitly.

Fix a block $T_k$, abbreviate $m=m_k$, and use the indices $a,b,c,d$ from Section~\ref{sec:mass}. Write $\sigma_a=\sigma_d=1$ and $\sigma_b=\sigma_c=-1$. For a generic sequence $q$, set $I_j=|q_j|^2$ and
\begin{equation}\label{eq:retained}
 \begin{aligned}
 S&=\sum_{j\in\Z}I_j,&Q&=\sum_{j\in\Z}I_j^2,\\
 F(I)&=2S^3-3SQ+\frac43\sum_{j\in\Z}I_j^3,&
 W&=q_aq_d\bar q_b\bar q_c+\bar q_a\bar q_dq_bq_c,\\
 J&=I_a+I_b+I_c+I_d,&C&=12S-6J,\\
 H_m&=\sum_{j\in\Z}\lambda_jI_j+\mu(F+CW),&A_m&=2S^2-Q+4W.
 \end{aligned}
\end{equation}
The part with zero exponent vector of $(1/3)\int|u|^6\,dx/(2\pi)$ is $F$. To verify its coefficients, three distinct indices give multiplicity $6\cdot6/3=12$, two equal indices give $3\cdot3/3=3$, and three equal indices give $1/3$. Hence this part is
\[
 12\sum_{r,s,t\in\Z,\ r<s<t}I_rI_sI_t
 +3\sum_{r,s\in\Z,\ r\ne s}I_r^2I_s+\frac13\sum_{r\in\Z}I_r^3
 =2S^3-3SQ+\frac43\sum_{r\in\Z}I_r^3.
\]
For the exponent vector $\sigma^k$, an additional action indexed outside $T_k$ has multiplicity $6\cdot6/3=12$; an additional action indexed in $T_k$ has multiplicity $3\cdot6/3=6$. Therefore the coefficient of the four-mode monomial is
\[
 12\sum_{j\in\Z\setminus T_k}I_j+6\sum_{j\in T_k}I_j=12S-6J=C.
\]
Adding the conjugate exponent vector gives $CW$. Similarly, the zero exponent vector in $\mathscr A$ gives $2S^2-Q$, while the four distinct indices in each monomial of $W$ give multiplicity $2\cdot2=4$. The retained part of $\mathscr A$ is consequently $A_m$.

For $j\in T_k$, differentiate with respect to the action first:
\[
 \frac{\partial F}{\partial I_j}=6S^2-3Q-6SI_j+4I_j^2,\quad
 \frac{\partial C}{\partial I_j}=12-6=6,\quad
 \partial_{\bar q_j}I_j=q_j.
\]
It follows that
\begin{equation}\label{eq:Pres}
 \begin{aligned}
 P_j^0
 &=\partial_{\bar q_j}(F+CW)-3A_mq_j\\
 &=\bigl(6S^2-3Q-6SI_j+4I_j^2+6W
             -3(2S^2-Q+4W)\bigr)q_j
       +C\partial_{\bar q_j}W\\
 &=(-6SI_j+4I_j^2-6W)q_j+C\partial_{\bar q_j}W.
 \end{aligned}
\end{equation}
This formula separates the retained nonlinear frequency correction from the term that exchanges the four actions. It also shows that each retained monomial contains at least three input factors with indices comparable to $m$.

For an output in $T_k$, the common-phase subtraction also cancels the interactions belonging to other blocks. For $l\ne k$, let $W_l$ be defined as $W$ in \eqref{eq:retained}, with $T_k$ replaced by $T_l$, and put
\[
 J_l=\sum_{r\in T_l}I_r,\quad C_l=12S-6J_l.
\]
The terms with exponent vectors $\pm\sigma^l$ in $(1/3)\int_{\T}|u|^6\,dx/(2\pi)$ are $C_lW_l$, and the corresponding part of $\mathscr A$ is $4W_l$. Since $j\in T_k$ lies outside $T_l$, we have
\[
 \partial_{\bar q_j}(C_lW_l)=12q_jW_l,
 \quad 3q_j(4W_l)=12q_jW_l.
\]
Their contributions to $P_j=N_j-3\mathscr A q_j$ cancel exactly. On the selected block, $\partial C/\partial I_j=6$ leaves the term $-6Wq_j$ in \eqref{eq:Pres}. The cancellation outside the selected block provides a random direction of fixed width for every nonretained divisor in Section~\ref{sec:potential}.

\subsection{Determining the homological transformation}

We remove $P_j^1$ from \eqref{eq:gaugedflow} to first order in $\mu$ by a correction supported on the four output coordinates,
\begin{equation}\label{eq:map}
 y=z+\mu G(z),\quad G_j=0,\quad j\notin T_k.
\end{equation}
We first determine the equation for $G$ using a finite polynomial correction, for which the coordinatewise chain rule is valid. Its derivative gives
\begin{align*}
 \dot y_j-i\lambda_jy_j
 &=i\mu P_j^0(z)+i\mu P_j^1(z)
     +\mu DG_j(z)[i\lambda z+i\mu P(z)]-i\mu\lambda_jG_j(z)\\
 &=i\mu P_j^0(z)
   +\mu\bigl(iP_j^1(z)+DG_j(z)[i\lambda z]-i\lambda_jG_j(z)\bigr)
   +\mu^2DG_j(z)[iP(z)].
\end{align*}
The terms in parentheses are exactly the unwanted terms of order $\mu$. Their cancellation requires the \textit{homological equation}
\begin{equation}\label{eq:homological}
 DG_j(q)[i\lambda q]-i\lambda_jG_j(q)=-iP_j^1(q).
\end{equation}
We solve this identity coefficientwise and then justify the resulting series along $z(t)$ by finite truncation and integration.

Here $DG_j(q)[h]$ denotes the real directional derivative
\[
 DG_j(q)[h]=\left.\frac d{dt}G_j(q+th)\right|_{t=0},\quad t\in\R.
\]
Thus positive factors differentiate to $h_n$ and conjugate factors to $\bar h_n$. For the monomial \eqref{eq:monomial}, the product rule reads
\begin{equation}\label{eq:monomial-derivative}
 DM_\nu(q)[h]
 =\sum_{r=1}^3 h_{n_r}
      \prod_{1\leqslant s\leqslant3,\ s\ne r}q_{n_s}
      \prod_{s=4}^5\bar q_{n_s}
 +\sum_{r=4}^5\bar h_{n_r}
      \prod_{s=1}^3q_{n_s}
      \prod_{4\leqslant s\leqslant5,\ s\ne r}\bar q_{n_s}.
\end{equation}
For $h_n=i\lambda_nq_n$, reality of $\lambda_n$ gives
$\bar h_n=-i\lambda_n\bar q_n$. Consequently,
\begin{equation*}
 DM_\nu(q)[i\lambda q]
 =i(\lambda_{n_1}+\lambda_{n_2}+\lambda_{n_3}
       -\lambda_{n_4}-\lambda_{n_5})M_\nu(q).
\end{equation*}
The term $-i\lambda_jG_j$ in \eqref{eq:homological} comes from subtracting the output rotation $i\lambda_jy_j$. It gives the contribution $-\lambda_j$ in the divisor $\Delta_\nu$.

Substituting $G_j=\sum_{\nu\in\mathcal R_j}g_\nu M_\nu$ into \eqref{eq:homological} now gives
$i\Delta_\nu g_\nu=-ic_\nu$ for each nonretained monomial. We therefore define
\begin{equation}\label{eq:G}
 G_j(q)=-\sum_{\nu\in\mathcal R_j}
        \frac{c_\nu}{\Delta_\nu}M_\nu(q),\quad j\in T_k,
 \quad G_j(q)=0,\quad j\notin T_k.
\end{equation}
For a finite set $\mathcal F\subset\mathcal R_j$, let $G_j^{\mathcal F}$ denote the corresponding finite sum of complete monomials. The cancellation is explicitly
\begin{equation}\label{eq:homological-cancellation}
 \begin{aligned}
 DG_j^{\mathcal F}(q)[i\lambda q]-i\lambda_jG_j^{\mathcal F}(q)
 &=-i\sum_{\nu\in\mathcal F}\frac{c_\nu}{\Delta_\nu}
   \bigl(\lambda_{n_1}+\lambda_{n_2}+\lambda_{n_3}
          -\lambda_{n_4}-\lambda_{n_5}-\lambda_j\bigr)M_\nu(q)\\
 &=-i\sum_{\nu\in\mathcal F}\frac{c_\nu}{\Delta_\nu}\Delta_\nu M_\nu(q)
 =-i\sum_{\nu\in\mathcal F} c_\nu M_\nu(q).
 \end{aligned}
\end{equation}

The inverse estimate \eqref{eq:inverse} gives absolute convergence and bounds for the correction. For every $p>3/2$, every fixed $R$, and $\|q\|_p\leqslant R$, it gives
\begin{equation}\label{eq:Gnorm}
 |G_j(q)|\leqslant C_{p,R}m^{-1}\la j^{-p},\quad
 |DG_j(q)[h]|\leqslant C_{p,R}m^{-1}\la j^{-p}\|h\|_p.
\end{equation}
Indeed, each undifferentiated factor contributes its weight times at most $R$, and each of the five terms in \eqref{eq:monomial-derivative} replaces one such factor by at most $\|h\|_p$ times the same weight. Absolute convergence gives a real differentiable map on $X_p$. In particular, \eqref{eq:map} changes the four coordinates by relative size $\mathcal{O}(\mu/m)$, which tends to zero on the parameter intervals used below.

\subsection{Implementation along the trajectory and the remainder}

We handle the unbounded linear part by applying \eqref{eq:homological-cancellation} to finite sums before passing to the series $G_j$.
Let $G_j^{(N)}$ retain the complete monomials in \eqref{eq:G} with $\max_r|n_r|\leqslant N$, and let $(P_j^1)^{(N)}$ be the matching truncation of \eqref{eq:decomposition}.
The coordinatewise chain rule and the finite homological cancellation give
\begin{align*}
 \frac d{dt}G_j^{(N)}(z(t))
 &=DG_j^{(N)}(z(t))[i\lambda z(t)]
       +\mu DG_j^{(N)}(z(t))[iP(z(t))]\\
 &=i\lambda_jG_j^{(N)}(z(t))-i(P_j^1)^{(N)}(z(t))
       +\mu DG_j^{(N)}(z(t))[iP(z(t))].
\end{align*}
The linear factors have combined into the fixed output frequency $\lambda_j$ and the polynomial $(P_j^1)^{(N)}$. We integrate this identity on $[0,t]$ and pass to the limit using the following tail bounds.

For each fixed output $j$, set
\[
 B_{j,N}=\sum_{\nu\in\mathcal R_j,\ \max_{1\leqslant r\leqslant5}|n_r|>N}
 \frac{|c_\nu|}{|\Delta_\nu|}
 \prod_{r=1}^5\la{n_r}^{-p}.
\]
By \eqref{eq:inverse}, $B_{j,N}\to0$. The uniform $X_p$ bounds for $z(t)$ and $P(z(t))$ give
\[
 \begin{aligned}
 \sup_{t\geqslant0}|G_j(z(t))-G_j^{(N)}(z(t))|
 &\leqslant C_{p,R}B_{j,N},\\
 \sup_{t\geqslant0}
 |(DG_j-DG_j^{(N)})(z(t))[iP(z(t))]|
 &\leqslant C_{p,R}B_{j,N}.
 \end{aligned}
\]
The convolution estimate \eqref{eq:convolution} gives the uniform $X_p$ bound for $P(z)$ and controls the tails of $(P_j^1)^{(N)}$ without inverse divisors. The coordinates of $z$ are continuous in time, and the $h^1$ convolution bounds give continuity of the coordinates of $P(z)$. Thus the finite expressions above are continuous and converge uniformly on each compact time interval. Passing to the limit in the integrated identity yields
\begin{align*}
 G_j(z(t))-G_j(z(0))
 =\int_0^t\bigl[&i\lambda_jG_j(z(s))-iP_j^1(z(s))\\
                 &+\mu DG_j(z(s))[iP(z(s))]\bigr]\,ds.
\end{align*}
The integrand is a uniform limit of continuous coordinate expressions, hence is continuous. The transformed coordinate is therefore differentiable, and its full derivative cancellation is
\begin{align}\label{eq:transformed-cancellation}
 \dot y_j
 &=\dot z_j+\mu\frac d{dt}G_j(z)\notag\\
 &=i\lambda_jz_j+i\mu P_j^0(z)+i\mu P_j^1(z)
   +i\mu\lambda_jG_j(z)-i\mu P_j^1(z)
   +\mu^2DG_j(z)[iP(z)]\notag\\
 &=i\lambda_j\bigl(z_j+\mu G_j(z)\bigr)
   +i\mu P_j^0(z)+\mu^2DG_j(z)[iP(z)]\notag\\
 &=i\lambda_jy_j+i\mu P_j^0(z)+\mu^2DG_j(z)[iP(z)].
\end{align}
Thus the unwanted term $P_j^1$ has been replaced by a term of order $\mu^2$.

The coordinate equations \eqref{eq:transformed-cancellation} and the bounds on $y-z$ provide the action and phase control along the hypothetical trajectory.

To apply the four-mode calculation, the retained vector field must be evaluated at $y$. Adding and subtracting $i\mu P_j^0(y)$ gives
\begin{equation}\label{eq:normal}
 \dot y_j=i\lambda_jy_j+i\mu P_j^0(y)+R_j,
 \quad j\in T_k,
\end{equation}
with the exact remainder
\begin{equation}\label{eq:exactresidual}
 R_j=\mu^2DG_j(z)[iP(z)]
          +i\mu\bigl(P_j^0(z)-P_j^0(y)\bigr).
\end{equation}
To estimate the two terms, the convolution bounds on a fixed $X_p$ ball give
\[
 \|P(z)\|_p\leqslant C_{p,R},\quad
 |P_j^0(z)-P_j^0(y)|
 \leqslant C_{p,R}\la j^{-p}\|z-y\|_p.
\]
Since \eqref{eq:Gnorm} yields $\|y-z\|_p\leqslant C_{p,R}\mu/m$, the two contributions satisfy
\begin{align*}
 \mu^2|DG_j(z)[iP(z)]|
 &\leqslant C_{p,R}\mu^2m^{-1}\la j^{-p}\|P(z)\|_p
 \leqslant C'_{p,R}\mu^2m^{-1}\la j^{-p},\\
 \mu|P_j^0(z)-P_j^0(y)|
 &\leqslant C_{p,R}\mu\la j^{-p}\|z-y\|_p
 \leqslant C'_{p,R}\mu^2m^{-1}\la j^{-p}.
\end{align*}
For sufficiently small $\mu/m$, both $z$ and $y$ lie in fixed balls, and all constants are independent of the chosen block. Finally, \eqref{eq:Pres} rewrites \eqref{eq:normal} as
\begin{equation}\label{eq:remainder}
 \begin{gathered}
 \dot y_j=i\partial_{\bar y_j}H_m(y)-3i\mu A_m(y)y_j+R_j,
 \quad j\in T_k,\\
 |R_j|\leqslant C_{p,R}\frac{\mu^2}{m}\la j^{-p}.
 \end{gathered}
\end{equation}
We next use the explicit retained interaction to estimate the actions and relative phase, with the remainder measured against \eqref{eq:required-error}.

\section{Phase rotation and a uniform action excursion}\label{sec:phase}

The detuning forces every initial relative phase to traverse a half-turn on which $\sin\Theta$ is positive. Bounds on both sides of the rotation speed ensure that this passage occurs in finite time and produces an action change exceeding the width of the action box.

\subsection{Exact equations for actions and relative phase}

The real common phase $-3\mu A_m(y)$ in \eqref{eq:remainder} cancels from the relative phase. We therefore consider
\begin{equation}\label{eq:forced}
 \dot y_j=i\partial_{\bar y_j}H_m+i\chi(t)y_j+R_j,
 \quad j=a,b,c,d,
\end{equation}
where $\chi$ is real. In the expressions \eqref{eq:retained}, take
$S=\sum_{j\in\Z}|y_j|^2$ and $Q=\sum_{j\in\Z}|y_j|^4$.
Write
\[
 y_j=\sqrt{I_j}e^{i\theta_j},\quad
 \Theta=\theta_a+\theta_d-\theta_b-\theta_c,\quad
 P=\sqrt{I_aI_bI_cI_d},\quad
 \Delta=\lambda_a+\lambda_d-\lambda_b-\lambda_c.
\]
Whenever the four moduli are bounded away from zero, the phases have continuous real lifts on the interval under consideration.
To identify the phase-dependent derivative, observe that
\[
 W=2P\cos\Theta,\quad
 \frac{\partial_{\bar y_j}W}{y_j}
 =\begin{cases}
 \bar y_a\bar y_dy_by_c/I_j=(P/I_j)e^{-i\Theta},&j=a,d,\\
 y_ay_d\bar y_b\bar y_c/I_j=(P/I_j)e^{i\Theta},&j=b,c.
 \end{cases}
\]
Thus this quotient is $(P/I_j)e^{-i\sigma_j\Theta}$ in both cases. The derivatives computed in \eqref{eq:Pres} also give
\[
 \frac{\partial_{\bar y_j}F}{y_j}
 =6S^2-3Q-6SI_j+4I_j^2,
 \quad
 \frac{\partial_{\bar y_j}(CW)}{y_j}
 =6W+C\frac P{I_j}e^{-i\sigma_j\Theta}.
\]
Dividing \eqref{eq:forced} by $y_j$ therefore yields
\begin{equation}\label{eq:coordinate-quotient}
 \frac{\dot y_j}{y_j}
 =i\bigl(\lambda_j+\Gamma-6\mu SI_j+4\mu I_j^2\bigr)
   +i\mu C\frac P{I_j}e^{-i\sigma_j\Theta}+\frac{R_j}{y_j},
\end{equation}
where $\Gamma=\mu(6S^2-3Q+6W)+\chi$ is real and independent of $j$.
Now use
\[
 \dot I_j=2I_j\operatorname{Re}\frac{\dot y_j}{y_j},\quad
 \dot\theta_j=\operatorname{Im}\frac{\dot y_j}{y_j},\quad
 ie^{-i\sigma_j\Theta}=\sigma_j\sin\Theta+i\cos\Theta.
\]
Taking the real part of \eqref{eq:coordinate-quotient} and the signed sum of its imaginary parts gives
\begin{align}
 \dot I_j&=2\sigma_j\mu CP\sin\Theta
                 +2\operatorname{Re}(\bar y_jR_j),\label{eq:actions}\\
 \dot\Theta&=\Delta+\mu\left[
 -6S\sum_{j\in\{a,b,c,d\}}\sigma_jI_j+4\sum_{j\in\{a,b,c,d\}}\sigma_jI_j^2
 +CP\left(\sum_{j\in\{a,b,c,d\}}\frac{\sigma_j}{I_j}\right)\cos\Theta\right]
 \notag\\
 &\quad+\sum_{j\in\{a,b,c,d\}}\sigma_j\operatorname{Im}(R_j/y_j).
 \label{eq:phase}
\end{align}
Since $\sigma_a+\sigma_b+\sigma_c+\sigma_d=0$, the entire common term $\Gamma$ cancels from \eqref{eq:phase}. We next bound the nonlinear terms in \eqref{eq:phase} to ensure that the detuning forces rotation.

\subsection{A reciprocal-action estimate}

The signed reciprocal term in \eqref{eq:phase} vanishes at equal actions. The following estimate retains this cancellation and bounds the term by the relative width of the action box.

\begin{lemma}\label{lem:reciprocal}
Let $0<d_0\leqslant1/2$, $l=1-d_0$, and $h=1+d_0$.
For $l\leqslant u_j\leqslant h$, $j=a,b,c,d$, set
\[
 P_0=\sqrt{u_au_bu_cu_d},\quad
 T=\sum_{j\in\{a,b,c,d\}}\sigma_ju_j,\quad U=\sum_{j\in\{a,b,c,d\}}\sigma_ju_j^2,\quad
 V=P_0\sum_{j\in\{a,b,c,d\}}\frac{\sigma_j}{u_j}.
\]
Then
\begin{equation*}
 P_0\geqslant l^2,\quad |T|\leqslant4d_0,
 \quad |U|\leqslant8d_0,\quad |V|\leqslant4d_0.
\end{equation*}
\end{lemma}
\begin{proof}
The first three bounds follow from the interval endpoints.
For the last bound,
\[
 \frac{\partial V}{\partial u_a}
 =\frac{P_0}{2u_a}
 \left(-\frac1{u_a}+\frac1{u_d}-\frac1{u_b}-\frac1{u_c}\right)
 \leqslant\frac{P_0}{2u_a}\left(\frac1l-\frac3h\right)\leqslant0,
\]
because $h\leqslant3l$. Similarly, $V$ is nonincreasing in $u_d$ and nondecreasing in $u_b,u_c$. Its maximum occurs at
$u_a=u_d=l$, $u_b=u_c=h$, and equals
\[
 lh\left(\frac2l-\frac2h\right)=2(h-l)=4d_0.
\]
Interchanging the positive and negative pairs gives the minimum $-4d_0$.
\end{proof}

\subsection{The escape estimate}

The following action box accommodates both the change of variables and the differences between adjacent reference weights. The detuning interval gives uniform escape for every initial phase.

\begin{lemma}\label{lem:escape}
Fix $I,S_*,\mu>0$ and put $\beta=\mu S_*I$.
Suppose \eqref{eq:actions}--\eqref{eq:phase} hold and
\begin{gather}
 0.699I\leqslant I_j\leqslant1.301I
                 ,\quad j=a,b,c,d,\label{eq:box}\\
 \left|\frac S{S_*}-1\right|\leqslant10^{-5},\quad
 \frac I{S_*}\leqslant10^{-5},\quad
 \max_j\frac{|R_j|}{\beta\sqrt I}\leqslant10^{-5},\notag\\
 21.87\leqslant D:=-\frac\Delta\beta\leqslant24.67.\label{eq:D}
\end{gather}
These conditions cannot hold throughout $[0,49/\beta]$.
\end{lemma}
\begin{proof}
Use the rescaled time $\tau=\beta t$, and set
$u_j=I_j/I$, $s=S/S_*$, $r=I/S_*$, and $J_0=\sum_{j\in\{a,b,c,d\}}u_j$.
With the notation of Lemma~\ref{lem:reciprocal}, we have
\[
 \begin{gathered}
 C=S_*(12s-6rJ_0),\quad P=I^2P_0,\\
 \sum_{j\in\{a,b,c,d\}}\sigma_jI_j=IT,\quad
 \sum_{j\in\{a,b,c,d\}}\sigma_jI_j^2=I^2U,\quad
 P\sum_{j\in\{a,b,c,d\}}\frac{\sigma_j}{I_j}=IV.
 \end{gathered}
\]
Dividing the action equation by $\beta I$ and the phase equation by $\beta$ gives
\begin{align*}
 u_a'&=2(12s-6rJ_0)P_0\sin\Theta+e_a,
 &|e_a|&\leqslant2\sqrt{1.301}\,10^{-5}<23\cdot10^{-6},
 \\
 \Theta'&=-D-6sT+4rU+(12s-6rJ_0)V\cos\Theta+e_\theta,
 &|e_\theta|&\leqslant\frac{4\cdot10^{-5}}{\sqrt{0.699}}.
\end{align*}
These bounds are uniform in the time-dependent quantities $s$ and the errors.
Apply Lemma~\ref{lem:reciprocal} with $d_0=0.301$.
Since $J_0\leqslant4(1.301)$, we have $12s-6rJ_0>0$ and
\begin{align*}
 6s|T|+4r|U|+|e_\theta|
 &\leqslant24(1+10^{-5})(0.301)
       +32\cdot10^{-5}(0.301)
       +\frac{4\cdot10^{-5}}{\sqrt{0.699}}
 <7.225,\\
 (12s-6rJ_0)|V|
 &\leqslant48(1+10^{-5})(0.301)<14.45,\\
 2(12s-6rJ_0)P_0
 &\geqslant\bigl[24(1-10^{-5})-48\cdot10^{-5}(1.301)\bigr](0.699)^2
 >11.72.
\end{align*}
It follows that
\begin{equation}\label{eq:phasespeed}
 -(D+7.225+14.45|\cos\Theta|)
 \leqslant\Theta'
 \leqslant-D+7.225+14.45\leqslant-0.195.
\end{equation}
By \eqref{eq:phasespeed}, the phase decreases strictly. For every initial phase, the next upper endpoint $\pi+2\pi n$ lies less than $2\pi$ below it, followed by a half-turn of length $\pi$ ending at $2\pi n$. If the hypotheses continue to hold, this passage is completed within normalized time $3\pi/0.195<49$. Let $\tau_1<\tau_2$ be the endpoints of the half-turn. Its duration is at most $\pi/0.195$.
On it, $\sin\Theta\geqslant0$. Changing variable from time to the decreasing phase and using
\[
 |\Theta'|\leqslant D+7.225+14.45|\cos\Theta|
 \leqslant31.895+14.45|\cos\Theta|
\]
gives
\begin{align}
 u_a(\tau_2)-u_a(\tau_1)
 &\geqslant11.72\int_{\tau_1}^{\tau_2}\sin\Theta\,d\tau
       -23\cdot10^{-6}(\tau_2-\tau_1)\notag\\
 &\geqslant11.72\int_0^\pi
 \frac{\sin\vartheta}{24.67+7.225+14.45|\cos\vartheta|}\,d\vartheta
       -23\cdot10^{-6}\frac\pi{0.195}\notag\\
 &=\frac{23.44}{14.45}
 \log\left(1+\frac{14.45}{31.895}\right)
       -23\cdot10^{-6}\frac\pi{0.195}.
 \label{eq:excursion}
\end{align}
For an exact lower bound, let
\[
 z=\frac{14.45}{2(31.895)+14.45}
   =\frac{1445}{7824}>\frac{923}{5000}.
\]
Using
$\log((1+z)/(1-z))\geqslant2(z+z^3/3)$,
$46.88/14.45>81/25$, and $\pi<22/7$, the right-hand side of
\eqref{eq:excursion} is greater than
\begin{equation*}
 \frac{81}{25}\left(\frac{923}{5000}
          +\frac13\left(\frac{923}{5000}\right)^3\right)
 -\frac{23}{10^6}\frac{22}{7}\frac{200}{39}
 >\frac{151}{250}=0.604.
\end{equation*}
However, \eqref{eq:box} allows at most
$1.301-0.699=0.602$ variation of $u_a$ between any two times.
This contradiction proves the lemma.
\end{proof}

\section{Completion of the proof}\label{sec:completion}

For the fixed multiplier, each sufficiently large block yields a parameter interval on which Lemma~\ref{lem:escape} excludes a trajectory contained in the tube. All estimates are uniform in $S_*\in[S_-,S_+]$. Overlap of these intervals will then cover every sufficiently small positive parameter with one threshold for all candidate trajectories.

\subsection{The parameter intervals}

To match the fixed detuning \eqref{eq:detuning} with the exchange scale, we centre the parameter interval at $\tau_k=23.2\mu S_*I_k$. For the fixed $M$ and the mass $S_*$ from \eqref{eq:mass}, define
\begin{equation}\label{eq:muwindows}
 I_k=m_k^{-M},\quad
 \mu_k(S_*)=\frac{\tau_k}{23.2S_*I_k},\quad
 \mathcal J_k(S_*)=[0.95\mu_k(S_*),1.05\mu_k(S_*)].
\end{equation}
If $\mu\in\mathcal J_k(S_*)$, then \eqref{eq:detuning} gives
\begin{equation}\label{eq:detuning-window}
 \frac{0.99(23.2)}{1.05}
 \leqslant\frac{-\Delta_k}{\mu S_*I_k}
 \leqslant\frac{1.01(23.2)}{0.95}.
\end{equation}
Its endpoints are respectively greater than $21.87$ and less than $24.67$, verifying \eqref{eq:D}.
Since $M<4\kappa+1/2<8\kappa$,
\[
 \mu\asymp(k+1)^2m_k^{M-8\kappa}\to0
\]
on these intervals, uniformly for $S_*\in[S_-,S_+]$.

\subsection{Action and mass bounds after the transformation}

The common-phase transformation preserves the actions and mass. For the second transformation, \eqref{eq:muwindows} gives
\begin{equation}\label{eq:uniform-map-smallness}
 \sup_{S_*\in[S_-,S_+],\ \mu\in\mathcal J_k(S_*)}
 \frac\mu{m_k}
 \leqslant\frac{1.05}{23.2S_-}(k+1)^2m_k^{M-8\kappa-1}
 \to0.
\end{equation}
Thus $y$ lies in the fixed larger ball used in \eqref{eq:remainder} for all sufficiently large blocks, with a starting index uniform in the candidate trajectory.

Apply \eqref{eq:map} to the block $T_k$ and put $m=m_k$. By \eqref{eq:Gnorm},
\[
 |y_j-z_j|\leqslant C\frac\mu m m^{-p},\quad j\in T_k,
\]
and all other coordinates remain unchanged. Hence
\[
 \bigl||y_j|^2-|z_j|^2\bigr|
 \leqslant2|z_j||y_j-z_j|+|y_j-z_j|^2
 \leqslant C\left(\frac\mu m+\frac{\mu^2}{m^2}\right)I_k.
\]
The total mass difference is at most four times this bound, since only these coordinates change. The original tube also gives
\[
 0.7\left(\frac m{m+r}\right)^M
 \leqslant\frac{|z_{m+r}|^2}{I_k}
 \leqslant1.3\left(\frac m{m+r}\right)^M,
 \quad r=0,1,2,3.
\]
The weight ratios tend to one. Uniformly in $S_*$, the lower-end loss from these ratios and the action error from the transformation are each at most $0.0005I_k$ for all sufficiently large blocks. Hence
\begin{equation}\label{eq:transformed-box}
 0.699I_k\leqslant|y_j|^2\leqslant1.301I_k,
 \quad j\in T_k.
\end{equation}
Because only four coordinates change and $S_*\geqslant S_->0$, we can also ensure
\begin{equation}\label{eq:transformed-mass}
 \left|\frac{S(y)}{S_*}-1\right|\leqslant10^{-5},
 \quad \frac{I_k}{S_*}\leqslant10^{-5}.
\end{equation}
Indeed, the relative mass error is at most
\[
 \frac{|S(y)-S_*|}{S_*}
 \leqslant\frac{4C}{S_-}
     \left(\frac\mu m+\frac{\mu^2}{m^2}\right)I_k\to0.
\]
The original conserved mass $S_*$ therefore remains the fixed scale for the transformed dynamics.

\subsection{The normalized remainder}

Take $p=M/2$ in \eqref{eq:remainder} and set $\beta=\mu S_*I_k$.
For $j\in T_k$,
\begin{align}\label{eq:closure}
 \frac{\max_{j\in T_k}|R_j|}{\beta\sqrt{I_k}}
 &\leqslant C_{\kappa,M}
 \frac{\mu^2m^{-1}m^{-M/2}}{\mu S_*m^{-3M/2}}
 \leqslant C'_{\kappa,M}\mu m^{M-1}\notag\\
 &\leqslant C''_{\kappa,M}(k+1)^2m^{2M-8\kappa-1}.
\end{align}
The comparison now identifies the precise restriction on $M$, i.e., $  2M-8\kappa-1<0 $ iff $  M<4\kappa+\frac12 $.
For this strict inequality, $m_k\asymp Q_\kappa^k$ implies
\[
 (k+1)^2m_k^{2M-8\kappa-1}\to0.
\]
Hence the error bound required in \eqref{eq:required-error} holds uniformly in time, in the hypothetical trajectory, and in $S_*\in[S_-,S_+]$.
Thus there is $K=K(\kappa,M)$ such that, for every $k\geqslant K$ and every $\mu\in\mathcal J_k(S_*)$,
\[
 \max_{j\in T_k}|R_j(t)|\leqslant10^{-5}\beta\sqrt{I_k},
 \quad t\geqslant0.
\]
Choose $K$ large enough to impose simultaneously the smallness requirement \eqref{eq:uniform-map-smallness}, the action and mass bounds \eqref{eq:transformed-box}--\eqref{eq:transformed-mass}, the normalized remainder bound, and $1.05\mu_K(S_-)\leqslant1$.
Equation \eqref{eq:remainder} has the form \eqref{eq:forced}, with the real common phase $\chi=-3\mu A_m(y)$.
Combining \eqref{eq:detuning-window}, \eqref{eq:transformed-box},
\eqref{eq:transformed-mass}, and \eqref{eq:closure}, Lemma~\ref{lem:escape} gives a contradiction within time $49/(\mu S_*I_k)$.
Every interval $\mathcal J_k(S_*)$, $k\geqslant K$, therefore rules out the hypothetical trajectory.

\subsection{Overlap and a uniform small-parameter threshold}

The fine geometric spacing in \eqref{eq:centers} makes the parameter intervals overlap and cover all sufficiently small positive parameters. By \eqref{eq:muwindows},
\[
 \mu_k(S_*)=\frac{(k+1)^2}{23.2S_*}m_k^{M-8\kappa}\to0,
\]
and
\begin{equation*}
 \frac{\mu_{k+1}(S_*)}{\mu_k(S_*)}
 =\left(\frac{k+2}{k+1}\right)^2
  \left(\frac{m_{k+1}}{m_k}\right)^{M-8\kappa}
 \to Q_\kappa^{M-8\kappa}
 \in(e^{-0.08},1)\subset(19/21,1).
\end{equation*}
The ratio is independent of $S_*$ and its limit exceeds $19/21=0.95/1.05$. After increasing $K$ uniformly in $S_*$, the centres decrease strictly and adjacent intervals overlap:
\[
 1.05\mu_{k+1}(S_*)>0.95\mu_k(S_*),\quad k\geqslant K.
\]
Since $\mu_k(S_*)\to0$, this overlap gives
\begin{equation*}
 \bigcup_{k\geqslant K}\mathcal J_k(S_*)
 =(0,1.05\mu_K(S_*)].
\end{equation*}
Since $\mu_K(S_*)$ is proportional to $S_*^{-1}$, its minimum over the allowed mass interval gives the common threshold
\begin{equation*}
 \mu_*(\kappa,M)=1.05\mu_K(S_+)
 =\frac{1.05\tau_K}{23.2S_+I_K}>0.
\end{equation*}
For every $S_*\in[S_-,S_+]$,
$\mu_K(S_*)\geqslant\mu_K(S_+)$.
Thus every $0<\mu<\mu_*(\kappa,M)$ belongs to at least one interval
$\mathcal J_k(S_*)$ associated with the mass of a hypothetical trajectory.
Lemma~\ref{lem:escape} therefore excludes every such trajectory. Only the choice of block depends on its conserved mass; the multiplier and the threshold are fixed.

\subsection{Amplitude rescaling}

For $A>0$, write $q=\sqrt A\,z$ in \eqref{eq:fourier}.
Quintic homogeneity gives the exact equation
\[
 \dot z_j=i\lambda_jz_j+i(\mu A^2)N_j(z),
\]
and
\[
 \frac{|q_j|^2}{A\la j^{-M}}=
 \frac{|z_j|^2}{\la j^{-M}}.
\]
The case $A=1$, with the same multiplier and effective parameter $\mu A^2$, therefore excludes forward complete trajectories under \eqref{eq:smallparameter} for every $A>0$. An invariant torus carrying a complete flow would supply such a trajectory. This proves Theorem~\ref{thm:main} assuming Proposition~\ref{prop:fixedpotential}, proved in the next section.

\section{Auxiliary estimates and selection of the multiplier}\label{sec:auxiliary}

We prove Proposition~\ref{prop:fixedpotential} by combining a weighted low-energy estimate with conditional fractional moments. These estimates yield one multiplier satisfying both the global arithmetic condition and the weighted inverse sum \eqref{eq:inverse} for all required decay exponents.

\subsection{Weighted summation over a low-energy window}\label{sec:window}

Large kinetic divisors give the factor $m^{-1}$ in \eqref{eq:inverse} directly. In the low-energy window, momentum and energy force additional large inputs whose polynomial weights provide the required gain. Estimating the whole window also controls the sum over its energy levels.

For $s>1$, the elementary convolution estimate is
\begin{equation}\label{eq:convolution}
 \sum_{a,b\in\Z,\ a+b=n}\la a^{-s}\la b^{-s}\leqslant C_s\la n^{-s}.
\end{equation}
Indeed, at least one of $|a|,|b|$ is at least $|n|/2$; extract its weight and sum the other one. Iteration gives the corresponding bound for every fixed number of factors, with arbitrary signs. In particular,
\[
 \|N(q)\|_p\leqslant C_p\|q\|_p^5,\quad
 \|DN(q)[h]\|_p\leqslant C_p\|q\|_p^4\|h\|_p.
\]
The same bounds hold for $P$ and its retained components $P_j^0$, which are sums of quintic momentum-preserving monomials with uniformly bounded multiplicities. The derivative bound on a ball gives the Lipschitz estimates in \eqref{eq:exactresidual}.

For five ordered integer inputs $\boldsymbol n=(n_1,n_2,n_3,n_4,n_5)\in\Z^5$ satisfying
$n_1+n_2+n_3-n_4-n_5=m$, put
\[
 D_0=n_1^2+n_2^2+n_3^2-n_4^2-n_5^2-m^2,
 \quad
 W_s(\boldsymbol n;m)=
 \frac{\prod_{i=1}^5\la{n_i}^{-s}}{\la m^{-s}}.
\]

\begin{lemma}\label{lem:window}
For every $s>1$, there are constants $C_s,m_s$ such that, for every $m\geqslant m_s$,
\begin{equation}\label{eq:window}
 \sum_{n_1+n_2+n_3-n_4-n_5=m,\ |D_0|\leqslant m/4,\ n_1,n_2,n_3\ne m}
 W_s(\boldsymbol n;m)
 \leqslant C_s\bigl(m^{1/2-s}+m^{2-2s}\bigr).
\end{equation}
\end{lemma}

\begin{proof}
Fix $\epsilon=1/100$. Momentum forces at least one input to have modulus greater than $\epsilon m$. With at least two such inputs, the energy relation forces a third, giving two high-frequency tails after one input is eliminated. With exactly one, its nonzero integer displacement from the output forces two smaller inputs to reach order $\sqrt m$.

Suppose first that at least two inputs have modulus greater than $\epsilon m$.
We claim that at least three inputs have modulus at least $c\epsilon m$, with an absolute $c>0$.
Order the moduli of all six indices, including the output, as
$N_1\geqslant\cdots\geqslant N_6$. Then
$N_1\geqslant m$ and $N_3\geqslant\epsilon m$.
If the three largest indices have the same energy sign, the energy relation implies
\[
 N_1^2\leqslant3N_4^2+m/4,
\]
which forces $N_4\geqslant cN_3$ for large $m$.
Otherwise, after reversing all signs if needed, write the three largest signed spatial indices as $x,y,z$, so that
\[
 x+y-z=t,\quad x^2+y^2-z^2=E,
 \quad |t|\leqslant3N_4,\quad |E|\leqslant3N_4^2+m/4.
\]
Each of $|x|,|y|,|z|$ is at least $N_3$, and
\[
 (x-t)(y-t)=\frac{t^2-E}{2}.
\]
If $N_4\leqslant N_3/100$, the left-hand side has modulus at least
$(0.97N_3)^2$, whereas the right-hand side has modulus at most
$6N_4^2+m/8$. Since $N_3\geqslant\epsilon m$, this is impossible for sufficiently large $m$.
Hence $N_4>N_3/100\geqslant\epsilon m/100$.
At most one of the four largest indices is the output, proving the claim.

Eliminate one of these three large inputs by the momentum relation.
Its weight divided by the output weight is bounded by a constant depending only on $s,\epsilon$.
At least two of the four remaining free inputs have modulus at least $c\epsilon m$. Dropping all other restrictions, their contribution is bounded by
\[
 C_{s,\epsilon}
 \left(\sum_{n\in\Z,\ |n|\geqslant c\epsilon m}\la n^{-s}\right)^2
 \left(\sum_{n\in\Z}\la n^{-s}\right)^2
 \leqslant C_sm^{2-2s}.
\]
There are only finitely many choices of the eliminated and large inputs.

Now suppose exactly one input has modulus greater than $\epsilon m$.
The unique large input must have positive energy sign; otherwise
$D_0\leqslant-m^2+4\epsilon^2m^2$, contradicting the energy window.
Call this input $z$, and denote the other four inputs by $a,b,c,d$, with two positive and two negative signs.
If $h$ is their signed sum, then
\[
 z=m-h,\quad |h|\leqslant4\epsilon m,\quad
 D_0=-2mh+h^2+E_0,\quad E_0=a^2+b^2-c^2-d^2.
\]
The exclusion of a positive input equal to $m$ implies the integer condition $h\ne0$.
Let $L$ and $L_2$ be the largest and second largest moduli among the four small inputs.
Since $|h|\geqslant1$,
\[
 20L^2\geqslant|h^2+E_0|\geqslant\frac{7m}{4},
 \quad L\geqslant\sqrt{\frac{7m}{80}}.
\]
If $L_2<L/10$, then
$0.7L\leqslant|h|\leqslant1.3L$ and $|E_0|\leqslant1.03L^2$.
Using $L\leqslant\epsilon m$, we would obtain
\[
 |D_0|\geqslant1.4mL-2.72L^2
       \geqslant(1.4-2.72\epsilon)mL>m/4,
\]
a contradiction. Therefore $L_2\geqslant L/10$.

Choose a largest small input, still denoted by $a$, as the variable to be solved for, and fix the other three small inputs. If $a$ has sign $\sigma\in\{1,-1\}$ and their signed sum is $h_0$, then $z=m-\sigma a-h_0$ and
\[
 D_0'(a)=-2\sigma m+2(1+\sigma)a+2\sigma h_0.
\]
On $|a|\leqslant\epsilon m$, $|h_0|\leqslant3\epsilon m$, this derivative has constant sign and modulus at least $1.9m$.
The preimage of the full interval $[-m/4,m/4]$ has length less than one, so at most one integer $a$ is allowed.
Let $b$ now denote a largest modulus among the three fixed inputs. Then
\[
 |b|\geqslant L/10\geqslant c\sqrt m,\quad
 |a|\geqslant|b|,\quad
 \frac{\la z^{-s}}{\la m^{-s}}\leqslant C_s.
\]
Thus $\la a^{-s}\leqslant\la b^{-s}$, and summing over the three fixed inputs gives
\[
 C_s\sum_{b\in\Z,\ |b|\geqslant c\sqrt m}\la b^{-2s}
          \left(\sum_{n\in\Z}\la n^{-s}\right)^2
 \leqslant C_sm^{1/2-s}.
\]
The finitely many choices of largest inputs and signs are absorbed in $C_s$.
Combining the two cases proves \eqref{eq:window}.
\end{proof}

\subsection{Simultaneous arithmetic and weighted inverse bounds}\label{sec:potential}

On the probability space \eqref{eq:centers}--\eqref{eq:potential}, the affine block relations separate the small detunings from the fixed-width random directions. Their integer kernel determines the exceptional case in the Diophantine estimate and the terms removed by the common-phase transformation.

The coefficients of $X_k,Y_k,Z_k$ in $\ell\cdot v$ are
\[
 \ell_{m_k}-\ell_{m_k+3},\quad
 \ell_{m_k+1}+\ell_{m_k+3},\quad
 \ell_{m_k+2}+\ell_{m_k+3}.
\]
All three coefficients vanish exactly when
\[
 \ell_{m_k}=\ell_{m_k+3},\quad
 \ell_{m_k+1}=\ell_{m_k+2}=-\ell_{m_k+3}.
\]
Outside the blocks, the coefficients vanish exactly when the corresponding $\ell_j$ vanish. Hence the fixed-width directions have the integer kernel
\begin{equation}\label{eq:kernel}
 \ell=\sum_{k=1}^{\infty}t_k\sigma^k,\quad
 t_k\in\Z,\quad (t_k)_k\text{ finitely supported}.
\end{equation}
A nonzero coefficient in any fixed-width direction is an integer of modulus at least one. The resulting uniform interval has length at least $1/50$.

\subsubsection{The product Diophantine condition}

To establish \eqref{eq:DC}, we use fixed-width variables outside the kernel and detuning variables on it. The factor $(k+1)^2$ in $\tau_k$ makes the latter probability bounds summable.

Since $\kappa\geqslant1$,
\[
 \sum_{t\in\Z\setminus\{0\}}\frac1{1+t^2\la j^{2\kappa}}
 \leqslant\la j^{-2\kappa}\sum_{t\in\Z\setminus\{0\}}t^{-2},
 \quad \sum_{j\in\Z}\la j^{-2\kappa}<+\infty.
\]
The Euler product therefore converges:
\[
 \sum_{\ell\in\Z^{\Z}_*}D_\kappa(\ell)^{-1}
 =\prod_{j\in\Z}\left(1+\sum_{t\in\Z\setminus\{0\}}\frac1{1+t^2\la j^{2\kappa}}\right)
 <+\infty.
\]
If a random variable is uniform on an interval of length $a$, then, for $0<e\leqslant1/2$,
\[
 \mathbb P\{\dist(V,\Z)<e\}\leqslant2e(1+3/a).
\]
For each exponent vector outside \eqref{eq:kernel}, condition on all variables except one fixed-width variable having a nonzero coefficient. Summing the resulting bounds with $e=\gamma_\kappa/D_\kappa(\ell)$ shows that the total probability of a violation among these exponent vectors is at most $C_\kappa\gamma_\kappa$.

Group the kernel exponent vectors \eqref{eq:kernel} by the largest block $k$ with $t_k\ne0$ and condition on all variables except $U_k$.
The variable $U_k$ appears with coefficient $t_k$, giving a uniform interval of length $|t_k|\tau_k/50$.
Set
\[
 b_k:=\sum_{t\in\Z\setminus\{0\}}D_\kappa(t\sigma^k)^{-1}.
\]
Since the blocks are disjoint,
\[
 D_\kappa(t\sigma^k)
 =\prod_{r=0}^3\bigl(1+t^2(m_k+r)^{2\kappa}\bigr)
 \geqslant |t|^8m_k^{8\kappa},\quad
 b_k\leqslant Cm_k^{-8\kappa},
\]
and the denominator factorizes over blocks:
\[
 D_\kappa\!\left(\sum_{l=1}^{\infty} t_l\sigma^l\right)
 =\prod_{l=1}^{\infty} D_\kappa(t_l\sigma^l).
\]
Summing over each earlier block gives $1+b_l$, including $t_l=0$. On the largest block the interval bound costs at most $C\tau_k^{-1}$, since $|t_k|\geqslant1$. Thus the total kernel bad probability is bounded by
\[
 C\gamma_\kappa\sum_{k=1}^{\infty}\tau_k^{-1}b_k\prod_{l=1}^{k-1}(1+b_l)
 \leqslant C_{\kappa,L}\gamma_\kappa\sum_{k=1}^{\infty}(k+1)^{-2}.
\]
Here $\prod_{l=1}^{\infty}(1+b_l)<+\infty$ by geometric growth of $m_l$, and
\[
 \tau_k^{-1}m_k^{-8\kappa}=(k+1)^{-2}.
\]
Choose $0<\gamma_\kappa\leqslant1/2$ so that the total arithmetic bad probability is less than $1/4$.
Because $\ell\cdot(j^2)_j\in\Z$, the complementary event is exactly \eqref{eq:DC}.

\subsubsection{A fixed-width direction for every nonretained term}

The common-phase cancellation gives every nonretained divisor a fixed-width random direction, allowing fractional inverse moments to be bounded uniformly in the output and independently of the shrinking widths $\tau_k$.

The exponent vector of a quintic output monomial has $\ell^1$ norm at most six.
On the kernel \eqref{eq:kernel},
\[
 \left\|\sum_{l=1}^{\infty}t_l\sigma^l\right\|_{\ell^1}=4\sum_{l=1}^{\infty}|t_l|.
\]
Thus a nonzero kernel exponent vector of such a monomial must be $\sigma^l$ or $-\sigma^l$ for one block $l$.
For an output $j\in T_k$, the exponent vectors $0,\pm\sigma^k$ have already been retained in \eqref{eq:decomposition}.
If $l\ne k$, multiply a term with exponent vector $\pm\sigma^l$ by $\bar q_j$.
Its degree-six monomial must be the corresponding degree-four block monomial times $|q_j|^2$: the exponent vector has four noncancelling factors and the remaining positive and negative factors have the same index, which must be $j$.
The original output monomial therefore contains exactly one positive occurrence of $q_j$.
Its coefficient vanishes by \eqref{eq:binomial}.
Every term in $P_j^1$ consequently has a nonzero coefficient in a fixed-width random direction.

For each nonretained divisor, condition on all variables except one fixed-width variable. The local integrability of $|x|^{-\vartheta}$ for $0<\vartheta<1$ gives
\begin{equation}\label{eq:fraction-general}
 \mathbb E|\Delta_\nu|^{-\vartheta}
 \leqslant\frac{100^\vartheta}{1-\vartheta}.
\end{equation}
Indeed, the effective interval length is $a\geqslant1/50$, and the integral of $|x|^{-\vartheta}$ over an interval of length $a$ is maximal when the interval is centred at zero. Its normalized integral is
$2^\vartheta a^{-\vartheta}/(1-\vartheta)$.
The bound is uniform in the interval centre and all remaining variables.

\subsubsection{A countable selection of weights}

To obtain \eqref{eq:inverse}, taking a $\vartheta$th power changes the input weight from $p$ to $p\vartheta$. The low-energy estimate and the target bound $m^{-1}$ lead to the exponents
\[
 \frac12-(p-1)\vartheta,\quad
 2-(2p-1)\vartheta.
\]
They are negative at $\vartheta=1$ for every $p>3/2$, and hence for suitable $\vartheta<1$. We first select one realization for a countable sequence $p_n\downarrow3/2$, choosing
\[
 p_n=\frac32+\frac1{4n},\quad
 \vartheta_n=1-\frac1{8n+2},\quad n\geqslant1.
\]
Fix $n$ and split the inverse sum in \eqref{eq:inverse}, with $p=p_n$, into a deterministically bounded part $F_{n,j}$ and a random part $V_{n,j}$.

First consider the terms with no positive occurrence of the output coordinate. They satisfy the exclusion $n_1,n_2,n_3\ne j$ in Lemma~\ref{lem:window}. Their kinetic divisor is $D_0$, and
\[
 |\Delta_\nu-D_0|\leqslant6\|v\|_\infty<7.
\]
For $|D_0|\geqslant j/4$ and sufficiently large $j$, we have
$|\Delta_\nu|\geqslant|D_0|-7\geqslant j/8$. The full normalized weight sum is bounded by the convolution estimate, so this part contributes $\mathcal{O}_n(j^{-1})$.
For $24<|D_0|<j/4$, the inverse divisor is at most one, and Lemma~\ref{lem:window} gives
\[
 C_n\bigl(j^{1/2-p_n}+j^{2-2p_n}\bigr)=\mathcal{O}_n(j^{-1}).
\]
Both exponents in the last display are less than $-1$ because $p_n>3/2$. These contributions therefore form $F_{n,j}\leqslant D_n/m_k$ for $j\in T_k$. The remaining sum $V_{n,j}$ consists of the small kinetic divisors without positive output factors and all surviving terms with at least two positive output factors.
For the finitely many outputs below the large-index thresholds, set $F_{n,j}=0$ and include the entire inverse sum in $V_{n,j}$, keeping the same $L$.
For $|D_0|\leqslant24$, take $j\geqslant\max\{96,m_{p_n\vartheta_n}\}$. Then $|D_0|\leqslant24$ implies $|D_0|\leqslant j/4$, so Lemma~\ref{lem:window} applies with exponent $p_n\vartheta_n$. The smaller outputs will be covered by the unrestricted convolution estimate.

For terms with no positive output occurrence and $|D_0|\leqslant24$, use the subadditivity inequality for $a_i\geqslant0$, $1\leqslant i\leqslant N$,
\[
 \left(\sum_{i=1}^N a_i\right)^{\vartheta_n}
 \leqslant\sum_{i=1}^N a_i^{\vartheta_n},\quad
 W_{p_n}^{\vartheta_n}=W_{p_n\vartheta_n},
\]
the fractional inverse bound \eqref{eq:fraction-general}, and Lemma~\ref{lem:window}.
Apply these estimates first to finite sets of monomials, then pass to the limit by monotone convergence and Tonelli's theorem. Since $|c_\nu|^{\vartheta_n}$ are uniformly bounded, the remaining sum has weight exponent $p_n\vartheta_n$, and its contribution to $\mathbb EV_{n,j}^{\vartheta_n}$ is bounded by
\[
 C_n\bigl(j^{1/2-p_n\vartheta_n}+j^{2-2p_n\vartheta_n}\bigr).
\]
For the remaining terms, select two positive output factors, leaving three inputs with momentum $a-b-c=-j$. With $s=p_n\vartheta_n>1$, division by the output weight cancels one of the two factors $\la j^{-s}$, so their normalized weight sum is at most
\[
 C_n\la j^{-s}
 \sum_{a,b,c\in\Z,\ a-b-c=-j}\la a^{-s}\la b^{-s}\la c^{-s}
 \leqslant C_n\la j^{-2s}.
\]
Consequently,
\begin{equation}\label{eq:nearprob}
 \mathbb EV_{n,j}^{\vartheta_n}
 \leqslant C_n\left(
 j^{1/2-p_n\vartheta_n}+j^{2-2p_n\vartheta_n}
                         +j^{-2p_n\vartheta_n}\right).
\end{equation}
For each of the finitely many earlier outputs, the fractional moment of the entire inverse sum is finite by \eqref{eq:fraction-general} and the unrestricted convolution estimate with exponent $s>1$. Enlarging $C_n$ therefore includes them in \eqref{eq:nearprob}.
Thus the inverse sums are finite almost surely, simultaneously over the countable family of outputs and weights $p_n$.

Markov's inequality for $V_{n,j}>H_n/m_k$ introduces $m_k^{\vartheta_n}$. Since $j\in T_k$ is comparable with $m_k$, the powers in \eqref{eq:nearprob} become
\[
 \frac12-(p_n-1)\vartheta_n=-\frac{6n+1}{8n(4n+1)}<0,\quad
 2-(2p_n-1)\vartheta_n=-\frac1{4n}<0,\quad
 -(2p_n-1)\vartheta_n=-2-\frac1{4n}<0.
\]
Geometric growth of $m_k$ yields
\[
 T_n:=\sum_{k=1}^{\infty}\sum_{j\in T_k}
            m_k^{\vartheta_n}\mathbb EV_{n,j}^{\vartheta_n}<+\infty.
\]
Choose $H_n\geqslant1$ finite with
$H_n^{-\vartheta_n}T_n\leqslant2^{-n-3}$.
For fixed $n$, Markov's inequality and a union bound give
\[
 \mathbb P\left\{V_{n,j}>\frac{H_n}{m_k}
                 \text{ for some }k,\ j\in T_k\right\}
 \leqslant H_n^{-\vartheta_n}
      \sum_{k=1}^{\infty}\sum_{j\in T_k}m_k^{\vartheta_n}
                         \mathbb EV_{n,j}^{\vartheta_n}
 \leqslant2^{-n-3}.
\]
Summing over $n$ bounds the probability of any of these failures by $1/8$.
Together with the arithmetic estimate, this leaves probability at least $5/8$ for the simultaneous conditions. Fix a realization in this event. It satisfies \eqref{eq:inverse} for every $p_n$.

For every real $p>3/2$, the momentum relation and $\la{n_i}\geqslant1$ give
\[
 \frac{\la j}{\prod_{i=1}^5\la{n_i}}\leqslant5.
\]
For $p\geqslant p_n$,
\[
 W_p(\boldsymbol n;j)
 \leqslant5^{p-p_n}W_{p_n}(\boldsymbol n;j).
\]
Every real $p>3/2$ is at least one of the $p_n$, so the same realization satisfies \eqref{eq:inverse} for all such $p$. This proves Proposition~\ref{prop:fixedpotential} and completes the proof of Theorem~\ref{thm:main}.

\end{document}